\documentclass[aap]{imsart}

\RequirePackage{amsthm,amsmath,amsfonts,amssymb}
\RequirePackage[colorlinks,citecolor=blue,urlcolor=blue]{hyperref}
\RequirePackage{graphicx}

\usepackage{comment}
\usepackage{pb-diagram}
\usepackage{colordvi}
\usepackage{graphics,eepic}
\usepackage[all]{xy}
\usepackage{mathrsfs}
\usepackage{amsfonts}
\usepackage{amssymb}
\usepackage{amsthm}
\usepackage[utf8]{inputenc}
\usepackage[english]{babel}

\usepackage{tikz}
\usepackage{tikz-cd}
 \usetikzlibrary{positioning,calc,arrows,shapes,automata,backgrounds,petri,trees,through,calc} 
\usepackage{pgfplots}
\usetikzlibrary{decorations.markings}

\startlocaldefs
\numberwithin{equation}{section}
\theoremstyle{plain}
\newtheorem{thm}{Theorem}[section]
\newtheorem{cor}{Corollary}[section]
\newtheorem{lem}[cor]{Lemma}
\newtheorem{prop}[cor]{Proposition}

\theoremstyle{remark}

\newtheorem{rem}{Remark}

\newcommand{\eps}{\varepsilon}

\newcommand{\N}{\mathbb{N}}
\newcommand{\E}{\mathbb{E}}

\newcommand{\R}{\mathbb{R}}

\newcommand{\bi}{\begin{itemize}}
\newcommand{\ei}{\end{itemize}}
\newcommand{\bd}{\begin{description}}
\newcommand{\ed}{\end{description}}
\newcommand{\beq}{\begin{equation}}
\newcommand{\eeq}{\end{equation}}
\newcommand{\beqn}{\begin{eqnarray}}
\newcommand{\eeqn}{\end{eqnarray}}
\newcommand{\beqna}{\begin{eqnarray*}}
\newcommand{\eeqna}{\end{eqnarray*}}

\def\P{{\mathbb P}}

\newcommand{\Var}{{\rm Var}}

\newcommand{\s}{{\big|}}

\newcommand{\bY}{{ \mathbf Y}}
\newcommand{\ind}{\mathbf{1}}

\newcommand{\lev}{ \mathcal L}      
\newcommand{\KRF}{Kac-Rice formula}

\newcommand{\grass}[2]{\mathbb{G}_{#1,#2}}
\def \2{\hbox{I}\!\hbox{I}}

\def \cs{\mathtt{(const)} }
\newcommand{\eq}{\begin{equation}}
\newcommand{\qe}{\end{equation}}

\newcommand{\crit}{{\rm Crt}}

\endlocaldefs

\begin{document}

\begin{frontmatter}
\title{On a general Kac-Rice formula for the measure of a level set}
\runtitle{Kac-Rice formula for  level set}

\begin{aug}

\author[A]{\fnms{Diego}~\snm{Armentano}\ead[label=e1]{diego.armentano@fcea.edu.uy}},
\author[B]{\fnms{Jean-Marc }~\snm{Aza\"is}\ead[label=e2]{jean-marc.azais@math.univ-toulouse.fr}}
\and
\author[C]{\fnms{José Rafael}~\snm{León}\ead[label=e3]{rlramos@fing.edu.uy}}
\address[A]{Instituto de Estadística, Facultad de Ciencias Económicas y de Administración,
Universidad de la Rep\'ublica, 11200 Montevideo, Uruguay\printead[presep={,\ }]{e1}}

\address[B]{Institut de Math\'ematiques de Toulouse, IMT, Université de Toulouse,
  France\printead[presep={,\ }]{e2}}
\address[C]{Insituto de Matemática y Estadística Rafael Laguardia, Facultad
de Ingeniería, Universidad de la Rep\'ublica, 11300 Montevideo,
Uruguay\printead[presep={,\ }]{e3}}
\end{aug}

\begin{abstract}
Let $X(\cdot) $ be a random field from $\R^D$ to $\R^d$, where $D\geq d$. We study 
  the level set $X^{-1}( u) $, where $u \in \R^d$. Specifically, we provide a weak
  condition for this level set to be rectifiable.  Next, we establish
  the Kac-Rice formula to compute the  $(D-d)$-dimensional Hausdorff measure. Our results
  extend previous work,   particularly in  the non-Gaussian case where we
  obtain a very general result. We
  conclude with several extensions and examples of applications, including functions of Gaussian random fields,  zeros of the likelihood functions, 
  gravitational  microlensing, and shot-noise.

\end{abstract}



\end{frontmatter}


\section{Introduction and motivation}
\subsection{History and present  of Kac-Rice formula}
The study of level sets began with the pioneering works of Kac
\cite{Kac} and Rice \cite{Rice} for random processes $ \R \to \R$.  Ito \cite{Ito} 
provided the first   rigorous  general proof of what has become  known as the
\emph{\KRF}.  A  first  comprehensive monograph on the subject was published by
Cramer and Leadbetter \cite{CL}.  At that  time, only
stochastic processes from $\R$ to $\R$ were considered.  The applications were
mainly  in engineering \cite {CL} or oceanography \cite{longuet0}.

In 1957, Longuet-Higgins \cite{longuet} published a paper on a model of
the sea surface using a Gaussian random field. In this paper, he derived
a \KRF~ for the expectation of the length of level set of such a
field. His proof was based on an intuitive approach with very
interesting ideas. Another notable work is a paper by Brillinger
\cite{Brillinger}, where he studied the zeros of a system of random
equations and,  for the first time, used the celebrated area formula.
Using this formula, Brillinger established the \KRF~ for almost every
level (see Proposition \ref{ae level:D:d} below),  noting that,  as long
as both sides are continuous functions of the level,  equality for all
levels is satisfied. However, his conditions for ensuring this
continuity are very strong; see his Corollary 3.1.

The first convincing application of the \KRF~to the expectation of the number of
zeroes of  random fields $\R^d \to \R^d$ was given by Adler \cite{adler}.
In fact, Adler considered critical points which are the zeroes of the
derivatives.  Two different general proofs were later published in \cite{AT}  and \cite{azais2005}.

Another generalization is the case  where dimensions differ: that is,
random fields $\R^D \to \R^d$,  with $D>d$. In this case,  the level set is
generically  a set  of dimension $D-d$. Defining   and  computing the expectation  of its  measure are  difficult problems. The first
attempts seem to  date back to Wschebor  \cite{wschebor1982} and
Caba\~{n}a \cite{cabana}.  The subject was further developed  by
Kratz-Le\'on \cite{KL}, \cite[Th 15.9.4]{AT} considered  the
case of i.i.d. Gaussian components and a first general proof was given 
in Aza\"is-Wschebor \cite{AW-book}, with a detailed version in Berzin et
al. \cite{berzin}. All the cited  works, with only marginal exceptions, are
limited  to the computation of the moments  of the measure of the level
set.  The computation of other  geometric quantities remains largely
an unsolved problem. 

In recent decades,  new applications  of the \KRF ~ have emerged. 
High-dimensional statistics and machine learning have
introduced the problem of studying the landscape of random fields, specifically the
behavior of the number of critical points with a given index on a
high-dimensional sphere.  This problem is closely related to the Ising model
\cite{auffinger2013}.

The random wave model, introduced by Berry,  along with the well-known Berry
cancellation problem \cite{Be}, have  been extensively studied in
physics  and astronomy.  A comprehensive review of the progress in the
study of nodal sets of random fields is available in the recent paper by
Wigman \cite{Wigman}.

In the 20th century, the Gaussian model was the most widely used. This was not
only because it avoided technical complications, but also because it was one of the few
models under which precise calculations could be performed. 
This explains why  most of the manuals focused primarily on this
model and  some  straightforward generalizations, such as $\chi^2$ processes.

In recent years, non-Gaussian models have been developed, including
non-Gaussian likelihood models, though the shot-noise model remains one
of the most common. For applications of the \KRF~  in this
latter process, see, for instance, the work of Biermé and Desolneux \cite{Bier:De}.

The cited literature provides numerous conditions  that
lead to various proofs of the Kac-Rice formulas.  In this paper, our
goal is to offer simple proofs under conditions that are both easier to
verify and weaker than those commonly used.  This is illustrated by
numerous examples  given in Section \ref{sec:examples}. 

\section{Main results}

Let $X:T\to\R^d$ be a random field defined on an open set
$T\subset\R^D$.

Given a fixed value $u\in\R^d$, and a Borel subset $B\subset T$, 
 we are interested in the level set 
\begin{equation}\label{def:levu}
  \lev_u(B) := \{ t\in B : X(t) =u\}.
\end{equation}

\medskip

In this paper, we  assume that $D\geq d$ (under general conditions, if
$D<d$, $\lev_u(B)$ is almost surely (a.s.) empty). We further assume that the paths of $X(\cdot)$ are of
class $C^1$ a.s., and we denote its derivative by $X'(\cdot)$.

\medskip

 Irregular points of the level set $\lev_u(B)$ are defined
 as points $t$ in $\lev_u(B)$, where the derivative $X'(t)$ is not a
 surjective linear map, that is when the generalized Jacobian  (or normal Jacobian) 
\begin{equation}\label{def:deltat}
 \Delta(t) := \sqrt{\det(X'(t) X'(t) ^\top)}
 \end{equation}
 vanishes. Here $X'(t)^\top$ denotes the transpose of the derivative.
In particular, when $D=d$, we have $\Delta(t)=|\det X'(t)|$: the absolute value of the
determinant of the Jacobian matrix of $X$ at $t$.



\medskip

A classical result, given by Bullinskaya, for a  one-parameter random field
is as follows:
\begin{lem}[Bulinskaya\cite{bulinskaya}]
 Let  $X: [a,b] \to \R$  be a random process  with paths of
 class $C^1$.  Let  $u\in \R$, and  assume that 
\begin{equation*}
p_{X(t)}(v)\leq C \mbox{ for all }t\in T \mbox{ and }v \mbox{ in some neighbourhood }V_u \mbox{ of }u. 
\end{equation*}
Then 
\begin{equation}\label{eq:realbuli}
\P \left\{ \exists t\in T,\, X(t)=u,\,X'(t)=0\right\}  =0.
 \end{equation}
\end{lem}

 The first main result  concerns the  Bulinskaya-type lemma
 concerning the regularity of the level set $\lev_u$ for $C^1$ random
 fields (see Proposition \ref{p:buli}).
Note that  this condition is required  in the proof of  the \KRF.  

Let us first consider the case $D=d$,  where the Bulinskaya
condition states that there are no  irregular points a.s. 
In classical results, certain assumptions are required to establish this condition, such as those in 
\cite[Lemma (11.2.1)]{AT} or \cite [Proposition 6.5]{AW-book}.
The present proof, based on local time, allows us to derive this
condition without additional assumptions. Moreover, even in the case
of very classical Gaussian random fields, our condition for the \KRF~ is weaker. In our opinion, the proof is also simpler.

Let us now consider the case where $D>d$.
Under additional smoothness conditions on the path of $X$, 
Morse-Sard type arguments and implicit function theorem
imply  that $\lev_u$ is a manifold for almost every level $u\in\R^d$.
However, our main interest lies in a specific value  of $u$, namely $u=0$.


It is sometimes stated  that if the paths of $X$ are $C^1$,
 the level set is always a manifold. However, this is only true under
 additional regularity conditions. The Morse-Sard theorem guarantees
 that the level set is a manifold for almost every level, provided the
 random field is sufficiently differentiable. This may have been the
 reasoning behind the assertion in the work of Ibragimov–Zaporozhets
 \cite{IZ}. Nevertheless, this statement does not hold in full generality, as
 demonstrated by the following counterexample:

Let $D=2$ $d=1$ and, for example, 
$$ 
X(t_1,t_2)  =  \xi ( t_1^2 - t_2^2),\qquad (t_1,t_2)\in\R^2,
$$
where $\xi$ is standard normal. Take $u=0$, the nodal set is a.s.  
$$
\lev_0 = \{ |t_1|= |t_2|\}\subset\R^2, 
$$ which is obviously not a manifold. \medskip

%
%
%
%
%

\medskip

If we assume the \emph{strong Bulinskaya condition:}
\begin{equation} \label{e:nocritical}
\P \left\{ \exists t\in T,X(t)=u,\Delta(t)=0\right\}  =0,
\end{equation}
then the  implicit function theorem implies that the level set $\lev_u$ is
a.s. a manifold.

Our contribution is to introduce the \emph{weak Bulinskaya condition}
\begin{equation} \label{e:buli:faible}
\sigma_{D-d}\{t\in T :\; X(t) =u, \, \Delta(t) =0 \} =0, \mbox{ a.s.}
\end{equation} 
where $\sigma_\ell$ denotes the $\ell$-dimensional Hausdorff measure on $\R^D$. Moreover,  we  provide weak sufficient conditions for \eqref{e:buli:faible} (Proposition  \ref{p:buli}).

(Note that $\sigma_0$ coincides with the counting measure; therefore,  conditions \eqref{e:nocritical} and
\eqref{e:buli:faible} are equivalent in the case $D=d$.)

Under condition \eqref{e:buli:faible}, we prove (as shown in Section
\ref{sec:buli}), that the level set $\lev_u$ is almost surely the union
of a $C^1$ manifold and a negligible set, and is therefore rectifiable.
\medskip

When $D>d$, classical proofs of \KRF~ (see \cite{AW-book}) require  the
strong condition \eqref{e:nocritical}. However, by using the Crofton
formula for rectifiable sets, we can prove the result under the  weaker
condition \eqref{e:buli:faible}, and once again, the local time argument
allows us to establish this  almost for free.
The underlying idea is as follows. 
Let $X(\cdot) $ be a ``nice'' random field. Then \eqref{e:buli:faible}
must hold,  because otherwise, the random field ``would spend too much
time in the neighborhood of $u$''. This is quantify by the
\emph{local time}: the density of the occupation measure  at $u$.  

This is the object of the next proposition.

\begin{prop}[Bulinskaya-type Lemma]\label{p:buli}
  Let $T\subset\R^D$  be an open  set and let  
     $X: T \to \R^{d}$  be a $C^1$ random field.  Let  $u\in \R^{d}$.
     Assume that the density $p_{X(t)}$ of $X(t)$ satisfies
  \begin{equation}\label{h:lt}
    \int_T p_{X(t)}(v)\,dt<C \mbox{ for all } v\mbox{ in some neighbourhood }V_u \mbox{ of }u. 
 \end{equation}
Then 
\begin{equation}
  \sigma_{D-d} \left(\{t\in T :\; X(t) =u, \, \Delta(t) =0 \}\right)
  =0,\;
  \mbox{a.s.}
 \end{equation}
\end{prop}


Proposition \ref{p:buli} significantly extends existing results
 (cf. Adler-Taylor \cite[Lemma 11.2.11]{AT}, Aza\"is-Wschebor \cite[Proposition 6.11]{AW-book}, and
  Nazarov-Sodin \cite[Lemma 6]{NaSo}.)
  Its proof is given in Section
 \ref{sec:buli}, where the local time is used.

 When $D>d$, since the regular part  of the level set  is a $C^1$
 manifold of dimension $D-d$,   it follows that  the level set is  a.s.
 a  $(D-d)$-rectifiable set.  In particular, its $(D-d)$-dimensional
 Hausdorff measure  coincides with the associated geometric measure of
 the regular part of $\lev_u$, i.e., the Riemaniann measure induced by
 the inherited Riemannian structure from $\R^D$ (see Morgan
 \cite{Morgan}, for example).

\bigskip

The next main result of this paper is a new version of the \KRF~  that 
extends to non-Gaussian random fields.  
While there are existing formulae for non-Gaussian processes in the
literature, such as \cite{Az}, \cite{Bier:De}, these typically rely on ad hoc proofs.
In the case of random fields, when $D=d$, Adler-Taylor \cite[Theorem
11.2.1]{AT} requires numerous conditions. The following theorem provides general results in the
non-Gaussian case.

 \begin{thm}[Rice formula for the expectation] \label{t:krf:dD}

Let $X : T  \to \R^d$ be a random field, $T$ an open subset of
   $\R^D$, ($D\geq d$). 

  Assume that:
\begin{enumerate}
  \item[(i)]  The sample paths of $X(\cdot)$  are a.s.  $C^{1}$,
 \item[(ii)] for each $t\in T$, $X(t)$ admits  a continuous density
   $p_{X(t)}(v)$, which is bounded uniformly  in $v \in \R^d$ and $t$ in any
    compact subset of $T$.
  \item[(iii)]  For {\bf every} $v \in \R^d$, for every $t\in T$,  the    distribution  of $\{ X(s), s\in T\}$ conditional to  $X(t) =v$  is well defined as a probability
    and is continuous, as a function of $v$, with respect to the  $C^1$ topology. 
\end{enumerate}
Then, for every Borel set $B$ contained in $T$ and for every level
   $u$, one has

\begin{equation} \label{e:krf:dD}  
\E \left(\sigma_{D-d}\left(\lev_u(B) \right) \right) =\int_{B} \E\left( \Delta(t) \s X(t)=u\right)\,p_{X(t)}(u)dt. \end{equation}
\end{thm}

\begin{rem}\label{extension} 
  In full generality, the conditions of
Theorem \ref{t:krf:dD} can be ``localized'':  it is sufficient for them
  to hold for  $X (t) $ in some neighborhood of the considered level
$u$.
\end{rem}

 \begin{rem}
The proof of Theorem \ref{t:krf:dD} is  partially based on the method of
   Aza\"is-Wschebor \cite{AW-book} and utilizes tools from integral geometry,
   such  as \emph{Crofton formula} for the rectifiable set $\lev_u$.
   (See Section \ref{sec:proofmain} for the proof and preliminaries on
   integral geometry.) 
 \end{rem}

\begin{rem}
The   $C^1$ topology is detailed in Section \ref{sec:proofmain}.
 
\end{rem}
\begin{rem}
Another way of presenting \eqref{e:krf:dD} is to say that the map $B \to
\E \left(\sigma_{D-d}\left(\lev_u(B)\right)\right)$  defines a Borel
measure, whose density  is given by $\E\left( \Delta(t) \s
X(t)=u\right)\,p_{X(t)}(u)$.  
\end{rem}

 \medskip

When $D=d$,  the  level set $\lev_u$, under  the conditions of this
theorem, consists of isolated points. Then, $\sigma_{D-d}\left(\lev_u(B)
\right)=N_u(X,B)$, where
\begin{equation}\label{def:nr}
N_u(X,B):=\#\{t\in B:\, X(t)=u\}
\end{equation}
is the number of solutions in $B$.
Then, formula (\ref{e:krf:dD}) reads as follows:
\begin{equation}\label{eq:KRF}
\E \left(N_u(X,B) \right) =\int_{B} \E\left( |\det(X'(t)| \s X(t)=u\right) ~p_{X(t)}(u)dt.
\end{equation}

\medskip

Theorem \ref{t:krf:dD} implies the following important
result for the Gaussian case. The key point  is that very few conditions are required. 

\begin{thm} \label{t:krf:gauss:dD}
Let $X : T  \to \R^d$ be a  Gaussian random field, $T$ an open subset of
$\R^D$, ($D\geq d$),  satisfying the following: 
\begin{enumerate}
 \item[(i)]  The sample paths of $X(\cdot)$  are a.s.  $C^{1}$;
\item[(ii)] for each $t\in T$, $X(t)$ has a positive definite
 variance-covariance matrix. 
\end{enumerate}

Then, \eqref{e:krf:dD} holds for every $u\in\R^d$. 
In addition, if $B$ is compact, both sides of \eqref{e:krf:dD} are finite and, consequently,
the measure $B\mapsto\E
\left(\sigma_{D-d}\left(\lev_u(B)\right)\right)$ is Radon.
 \end{thm} 

\begin{rem}
  Suppose that $(ii)$ is not met on $T$. Define 
  $$
  \tilde{T}:=\{ t\in T:\, \Var(X(t))\,\mbox{positive definite}\},
  $$
  which is an open set. 
 Then, \eqref{e:krf:dD} is valid for all Borel subsets of $\tilde{T}$.
\end{rem}

 \begin{rem} 
  This theorem  is a direct consequence  of Theorem  \ref{t:krf:dD}  once the conditional distributions are defined by means of regression formulas. 
\end{rem}

\begin{thm} \label{c:foncgauss}
Suppose now that $X$   is not Gaussian but a function of a Gaussian
  random field $Y$. More precisely,
$$
X(t)  = H\big(Y(t)\big). 
$$
 We assume the following for $v$ in a neighbourhood of the considered
  level $u\in\R^d$.
\begin{itemize} 
 \item[a)]$Y(\cdot) $ is a Gaussian random field  with  parameter in
   $\R^D$ and values in $\R^n$, $n\geq d$ with $C^1$ paths  and 
has a positive definite variance-covariance matrix.
   \item[b)] $H (\cdot):\R^n\to\R^d$ is a $C^1$ function.
   \item [c)] For each $v$, the  irregular set 
 $$
 \{ y \in  H^{-1} (v) :  \det(H'(y) H'(y)^\top ) =0\}$$
 is empty. 


%
 \end{itemize}
   Then, the \eqref{e:krf:dD} holds true for every Borel set $B$ contained in
   $T$.
\end{thm}

\begin{rem} The proof of this theorem is given in Section
  \ref{sec:examples}, where we also present  interesting examples, most
  of which are non-Gaussian.  In such case,   the main difficulty lies
  in defining a regular version of the conditional expectation.  One of
  the main contribution is to define it through a  very general version of the
  push-forward distribution, Lemma \ref{p:push-forward:g}, based on the
  co-area formula. This approach allows us to define the aforementioned conditional
  expectation in a variety of situations.  
\end{rem}












%
%
%
%
%
%
%

\subsection{\KRF~ for almost every level}

The following proposition is a version of the \KRF~ 
for almost every level. 

\begin{prop}\label{ae level:D:d} [\KRF~ for almost every level]
Let $X : T \to \R^d$ be a random field, and $T$ an open subset of
$\R^D$ . Assume the following:
  \begin{enumerate}
    \item[(i)] $ X(\cdot)$  has $C^1$ paths;
    \item[(ii)] for each $t \in T$ : $ \Delta(t)$ is integrable; 
    \item[(iii)] for each $t \in T$ :  $X(t)$ admits a density $p_{X(t)}(\cdot)$  in $\R^{d}$.
  \end{enumerate}
Then,  for every  Borel set  $B \subset T$, we have 
\begin{equation}\label{eq:riceas}
  \E (\sigma_{D-d} (\lev_u(B)))=\int_{B}~ \E \big( \Delta(t) \s X(t)=u
  \big)\,p_{X(t)}(u)~dt
\end{equation}
{\bf for almost every} $u \in \R^d.$
  \end{prop}
This result is simpler than Theorem \ref{t:krf:dD} and follows
almost directly from the co-area
formula. 
This is well known from the works of Brillinger \cite{Brillinger}, and
Z\"ahle \cite{Zahle} and we provide the proof for completeness.

\begin{proof}
 Let $g:\R^d\to\R$ be a test function, a bounded Borel non-negative function. 
  By the co-area formula \cite[Proposition 6.13]{AW-book}, one gets
\[
  \int_{\R^d} g(u) \sigma_{D-d}\left(X^{-1}(u)\cap B \right)du = \int_B
  g(X(t)) \,|\Delta(t)|\, dt.
\]
 Take  expectations on both sides. 
 Note that the expectations are always well-defined because we consider non-negative functions,
and also that the conditional expectation  
 $\E \big(|\Delta(t)| /X(t) =u \big)$ is well-defined for almost every $u$. 
 Applying  Fubini's theorem and conditioning on $X(t) =u$, we get 
\[
  \int_\R   g(u) \E\big(\sigma_{D-d}(\lev_u(B))\big) du = \int_\R  g(u)
  du  \int_B \E \big(|\Delta(t)| \s X(t) =u \big) p_{X(t)} (u) dt.
\]
 By duality, since $g(\cdot)$ is arbitrary,  the  terms in front of the
  factor of $g(u)$ on both sides are equal for almost every $u$, which
  leads to the result.
\end{proof}

Some practitioners consider the result above sufficient for applications. 
However, this is rarely the case. Often, the considered level 
is $u=0$ (as in the seminal work of Kac \cite{Kac}), and a statement
that holds for almost every $u$ does not necessarily hold   for $u=0$. 
It is therefore necessary to establish formula \eqref{eq:riceas} {\bf for every
level}, which requires considerable effort, especially for non-Gaussian models.

\medskip

To prove the continuity of the conditional expectation $\E
\big(\Delta(t) \s X(t) =u \big)$ with respect to the level $u$, the
proof is  straightforward in the Gaussian case, using regression formulas.  
In the general setting, the right-hand side of \eqref{eq:riceas} is defined, a
priori, only for almost every $u$.  Additional work is required to define the
density and the conditional expectation for all $u$.  Without this
effort,  the \KRF~   for all levels would not make sense.

\section{Proof of Proposition \ref{p:buli}}\label{sec:buli}

Let us start recalling some definitions.

Given $x\in\R^\ell$ and $\rho>0$, we denote by 
$B_\ell(x,\rho)\subset\R^\ell$ the ball centered at $x$  and of radius
$\rho$.

For a given $K\subset\R^D$, recall that $\sigma_\ell(K)$ is the
$\ell$-dimensional Hausdorff measure given by
$$
\sigma_\ell(K)=\lim_{\eps\to 0}  \sigma_{\ell} ^\eps (K),
$$
where $\sigma_{\ell} ^\eps$ is the $\ell$-dimensional Hausdorff
pre-measure on $\R^D$ given by
\begin{equation}\label{def:premeasure}
   \sigma_{\ell} ^\eps(K) := \alpha_\ell\inf\left\{ \sum_i
   {\rho_i}^{\ell}, \mbox{ for a
   covering}\;\{B_D(x_i,\rho_i)\}\;\mbox{of}\; K,\; \rho_i < \eps\right\},
\end{equation}
where $\alpha_\ell:=\lambda_\ell(B_\ell(0,1))$ is the Lebesgue measure of the unit ball
in $\R^\ell$. With this constant, ones get that the 
$\ell$-dimensional Hausdorff measure on $\R^\ell$ coincides with the
Lebesgue measure $\lambda_\ell$ on Borel sets of $\R^\ell$. 

\medskip

Before proving Proposition \ref{p:buli}, we need a technical lemma.
  \begin{lem} \label{l:diego}
   Let $K$   be a compact set in $ \R^D$,  and assume  $d<D$.  Suppose
    $\sigma_{D-d}(K)$ is finite. Then, there exist constants $C_1$ and
    $C_2$, depending on $d$ and $D$,  such that, for sufficiently small $\eps $:
 \begin{itemize}
   \item  there exists an ``$\eps$-packing'' (a collection of disjoint balls with center  belonging to
     $K$ and with  radius $\eps$) with cardinality $l(\eps)$,
 $$ l(\eps) \geq C_{1} \eps^{d-D} \sigma_{D-d}(K);
 $$
 \item 
   the parallel set $K^{+\eps}=\bigcup_{x\in K}B_D(x,\eps)$  satisfies
 $$
 \sigma_D(K^{+\eps})\geq  C_{2} \eps^{d}
     \sigma_{D-d}(K).
$$
 \end{itemize}
\end{lem}
\begin{proof}
  Recall that the premeasure $\sigma_{D-d}^\eps$, given in
  (\ref{def:premeasure}), satisfies $ \sigma_{D-d} ^\eps(K) \uparrow
  \sigma_{D-d} (K)$ as $\eps \downarrow 0$.
  Suppose $ \sigma_{D-d} (K) =a >0$ (otherwise, there is nothing to
  prove). Then, for  sufficiently small $\eps$, $ \sigma_{D-d} ^{3\eps}(K) >a/2$.
    We perform  a covering of $K$  by induction,  using balls  centered
    at  points in $K$  with  radius  $3\eps$.  To do so, we start with 
    one point of $K$,  and at each step,  we define a new ball  such
    that (i) its center belongs to $K$ and  (ii) it does not belong to the
    union of the preceding balls.  By compactness, the needed number of
    such  balls, denoted  $\ell(\eps)$, is finite. Moreover,  by definition of the
    Hausdorff pre-measure, 
    $$ 
   \ell (\eps) (3\eps)^{D-d}\alpha_{D-d} > a/2. 
   $$
   Let $t_1, \ldots, t_{\ell (\eps)}$  be the centres of the balls;
   the balls $B_D(t_i, \eps) $ constitute  an $\eps$  packing
   which proves the first assertion.  The  union of the balls, which has a
   volume not less than
   $$
  C_{2} a \eps^d
  $$
  (for certain constant $C_{2}$),  is contained whithin the parallel set  $K^{+\eps}$, which proves  the second assertion. 
  \end{proof}

Now, we turn to the proof of Proposition \ref{p:buli}.

\begin{proof}[Proof of Proposition \ref{p:buli}]
To obtain the result it is sufficient to prove \eqref{e:buli:faible} for every  compact $K$ contained in  $T$.
  Assume, without loss of generality, that  $u=0$.

   Define the density of the occupation measure  $\textrm{LT}$, the
   \emph{local time}, as 
   $$
    \textrm{LT} :=  \liminf_{\delta \to 0}  \textrm{LT}(\delta) :=
    \liminf_{\delta \to 0}\frac{1}{\lambda_d(B_d(0,\delta)}
    \sigma_D\left(\{t\in T: \|X(t)\| \leq \delta \}\right)  \ , $$
  where $\lambda_d(B_d(0,\delta))=\alpha_d\,\delta^d $ is the volume of
  the ball $B_d(0,\delta)$ with radius $\delta$ on $\R^d$. Note that
  since we used $\liminf$,  the definition  of the local time is in a
  generalized sense. 
By Fubini's theorem, the hypothesis on the density and Fatou's lemma:
  \begin{equation}\label{eq:espLT}
    \E   \big(\textrm{LT}\big) \leq C.
  \end{equation}

  This implies that $\textrm{LT}$ is an almost surely finite non-negative random variable.

  \medskip

 Let $M:=\max_{t\in T}\lambda_{max}(X'(t))$, where $\lambda_{max}$ is the
greatest singular value. 
Let$$
N(\eps):= \sup_{t\in
T;\,\|v\|<\eps}\frac{\|X(t+v)-X(t)-X'(t)v\|}{\|v\|}.
$$
Then, by compactness, $N(\eps) $ and $M$  are
  almost surely finite non-negative random variables and 
  $N(\eps)\to 0$, as $\eps\to 0$.

  Let $\tilde\lev_0$ be the irregular part of the level set $\lev_0$, i.e.,
  \begin{equation}\label{def:criticallevel}
  \tilde\lev_0:=\{t\in T:\,X(t)=0,\, \Delta (t)=0\}.
  \end{equation}

  Let  $t\in\tilde\lev_0$, then $X(t) =0$, $\textrm{rk}(X'(t))=k$ for
  some $k\in\{0,1,\ldots,d-1\}$. 
 
  Let $v_1\in V_1:=\ker (X'(t))$, where $V_1$ has dimension $D-k$.
\\
Let $v_2\in V_2:=V_1^\perp$, where $V_2$ has dimension $k$.

Then we  have 
    \begin{align*}
  \left\| X(t+v_1+v_2)\right\|&\leq 
  \left\| X(t+v_1+v_2) - X(t+v_1)  \right\|+ \left\|
  X(t+v_1)\right\|\\
  &\leq M\|v_2\|+ \|v_1\|\cdot N(\|v_1\|).
\end{align*}
  
  Suppose $\eps$ is such that $N(\eps)<1$, choose $v_1$
and $v_2$ such that
\begin{equation} \label{e:polydisk:1}
\|v_1\|\leq \frac{1}{\sqrt{2}} \eps,\quad \|v_2\|\leq \frac{1}{\sqrt{2}} 
\eps\cdot N(\eps).
\end{equation}
Then we have
$$
\left\| X\left(t+v_1+v_2\right)\right\|\leq \frac{1}{\sqrt{2}} 
  (M+1)\eps\cdot N(\eps).
$$
The condition given in \eqref{e:polydisk:1} defines a polydisk with volume 
  \begin{equation}\label{eq:volply}
 \cs \eps^{D-k} \big( \eps\cdot N(\eps)) ^k,
  \end{equation}
where the constant  depends on $D,\,d$, and $k$. The polydisk  is  included in
$\bar \lev_0^{+\eps}$ and, thus, in $T$ for   sufficiently small $\eps$.

  \medskip

Let us define the event $Z_{D,d}=\{\omega:\,\sigma_{D-d}(\tilde\lev_0)>0\}$,
and assume, by contradiction, that $\mathbb{P}(Z_{D,d})>0$.

In the case $D=d$, given $\omega\in Z_{d,d}$, let
$\delta=\delta_{\omega,\eps}=\big((M+1)\eps\cdot N(\eps)\big)$. 
Then, the approximated local time $\textrm{LT}(\delta)$ is greater than
$$
 \frac{1}{(\eps\cdot N(\eps))^d}\cs \eps^{d-k}
 \big( \eps\cdot N(\eps)) ^k=\cs (N(\eps))^{k-d},
$$
which tends to infinity as $ \eps \to 0$. This gives a contradiction
with the fact that $\textrm{LT}$ is an almost surely finite non-negative
random variable (cf. (\ref{eq:espLT}) for $D=d$).

\medskip

Now, let us consider the case $D>d$. When $\omega\in Z_{D,d}$, by  Lemma
\ref{l:diego} and (\ref{eq:volply}),  the approximated local time  $
\textrm{LT}(\delta)$ for $\delta=(M+1)\eps\cdot N(\eps)/\sqrt{2} $ satisfies
\begin{align*}
  \textrm{LT}(\delta) &= \frac{1}{\lambda_d(B_d(0,\delta)}
  \sigma_D\{t\in T: \|X(t)\| \leq (M+1)\eps\cdot N(\eps)/\sqrt{2}
  \} \\
  & \geq 
  \frac{\cs}{\lambda_d(B_d(0,\delta)} \ell(\eps)\eps^{D-k}
  \big( \eps\cdot N(\eps)) ^k\\
  &\geq 
 \cs ( \eps  N(\eps)) ^{-d} \eps^{D-k} \big(
 \eps\cdot N(\eps)) ^k\times \eps^{d-D} \\
  &= \cs (   N(\eps) ^{k-d}, 
\end{align*}
which again tends to infinity as $ \eps \to 0$, leading to a
contradiction.
\end{proof}

\section{Proof of Theorem \ref{t:krf:dD} for $D=d$}\label{sec:proofmain}

We first prove  Theorem \ref{t:krf:dD} for the case $D=d$. The general
  version of the proof, which follows from this case and makes use of the
  Crofton formula, will be posponed to Section \ref{sec:KRFD}.

Let  $C^1(T,\R^d)$ be the space of $C^1$ functions defined on $T$ with
values on $\R^d$. 
This space is equipped  with  the following canonical norm:
     \begin{equation}\label{def:norm1}
     \|f\| = \sup_{t \in T} ( \|f(t)\| +  \|f'(t)\|),
     \end{equation}
     where we have chosen, once and for all, a certain norm in  $\R^d$
     and in the space  of $(d\times d)$-matrices.

The following proposition, concerning the continuity of the number of roots
with respect to the $C^1$ topology, is inspired by the proof in
dimension 1  by Angst-Poly \cite{angst2019zeros}.

\begin{prop}\label{p:poly} Let $T$ be a compact set of $\R^d$.
 Suppose $X_n(\cdot), n=1,2,\ldots$  is a sequence  of random
  fields  defined on $T$  and belonging to the space $C^1(T,\R^d)$
  defined in (\ref{def:norm1}). Suppose that the sequence $X_n(\cdot)$ converges a.s.  to
 $X(\cdot)$  with the aforementioned topology.  Suppose that in
 addition, a.s.  \begin{itemize}
   \item[(a)] $N_{u}(X,\partial T) =0$ ,
   \item[(b)] $\{t \in T : X(t) = u ; \  \Delta(t) =0\} =\emptyset.$ 
   \end{itemize} 
    Then
     $$
     N_{u}(X_n,T)  \to  N_{u}(X,T)\quad   \mbox{a.s. as}\quad n \to \infty.
$$
\end{prop}
Note that we obtain the same result if we replace  the a.s. convergence by the  weak convergence.

  \begin{proof}   
 In this proof, we use the operator  norm for $\|X'(t)\|$. 
    For each realisation, let
    $m$  be the number of solutions  of $X(t) =u$ in $T$.  Since $T$ is
    compact,  (b) implies that $m$ is finite.  The case  $m=0$ is direct.    We
    consider  the case $m>0$. 

    Let $x_1,\ldots,x_m$ be the solutions to
    $X(t)=u$ for $t\in T$.
    By hypothesis, these solutions are interior points of $T$. 
  
     By the inverse function theorem, 
    for sufficiently small  $\delta$ and for all $i=1,\ldots,m,  X(\cdot)$ is a diffeomorphism   from $ B(x_i,\delta)$ to some open set of $\R^d$. Define
    $$
    W:=\cup_{i=1}^m  B(x_i,\delta).
    $$
   We may assume that  there exists $\alpha>0$ such that
     $\|(X'(t))^{-1}\|\leq \alpha$, for all $t\in W$.

    By uniform convergence in the $C^1$ topology, there exists  $N\in\N$
    such that 
   \begin{equation}\label{eq:lemaXnXN}
     \|X_n-X\|<\min(\delta,(4\alpha)^{-1}),\quad \mbox{for all}\quad n>N
   \end{equation} 
    and   $X_n(\cdot) $ has no solution to $X_n(t)=u$   on $T\setminus W$ for $n>N$.

%

    Let us define the auxiliary  $ C^1$ functions   $\widetilde X_{n,i}: 
    B(x_i,\delta)\to\R^d$, given by
    $$
   \widetilde X_{n,i}(t)= t - (X'(x_i))^{-1}( X_n(t)-u),\qquad (i=1,\ldots,m).
   $$
 Note that 
    $X_n$ has a solution to $X_n(t)=u$ on $B(x_i,\delta)$ if and only if 
 the map $\widetilde X_{n,i}$ has a fixed point.
 
  By strengthening the conditions on  $\delta$, we can even require the following

 \begin{equation}\label{eq:lemaXnXVi}  
   \sup_{t\in B(x_i,\delta)}  \|X'(t)-X'(x_i)\|<
  (4\alpha)^{-1},\qquad (i=1,\ldots,m).
  \end{equation}
    From
    \eqref{eq:lemaXnXN} and \eqref{eq:lemaXnXVi}, for all $t\in
    \widetilde B(x_i,\delta)$, 
 \begin{align*}
   \|(\widetilde X_{n,i})'(t)\|& =\|I_d - (X'(x_i))^{-1}X_n'(t)\|\\
   &\leq \|(X'(x_i))^{-1}\|\cdot \|X'(x_i)-X_n'(t)\|\\
   &\leq \alpha\cdot 
   \left(\|X'(x_i)-X'(t)\| + \|X'(t)-X_n'(t)\|\right)\\
   &< \alpha\left( \frac{1}{4\alpha}+ \frac{1}{4\alpha}\right)=
   \frac12,
 \end{align*} 
for all $n>N$.
    We conclude that $\widetilde X_{n,i}$ is a contraction $
    B(x_i,\delta) \to B(x_i,\delta)$. The fixed point theorem  implies 
     that  $\widetilde X_{n,i}$ has exactly  one fixed point. 
    
     So, for large enough $n$, the number of solutions is $m$,
     concluding our proof.
    \end{proof}

We also need  another version of the Bulinskaya Lemma  which  is
extracted  from  \cite[p. 13]{AW-book}.
  \begin{lem}\label{l:bufield}
Let $\mathcal{Y}=\{Y(t): t\in W \}$ be a random field with values in
$\R^{m+k} $, where $W$ an open subset of $\R^d$, and  $m$ and $k$ are
positive integers. Let $u\in \R^{m+k}$ and $I$ a subset of $W$.\\
We assume that $\mathcal{Y}$ satisfies the following conditions:
\begin{itemize}
\item the paths $t\mapsto Y(t)$ are of class
$C^1$,
\item for each $t\in W$, the random vector $Y(t)$ has a density and
there exists a constant $C$ such that
$$
p_{Y(t)}(x) \leq C
$$
for $t\in I$ and $x$ in some neighborhood of $u$,
\item the Hausdorff dimension of $I$ is smaller or equal than $m$.
\end{itemize}
Then, almost surely, there is no point $t\in I$ such that $Y(t)=u$.
\end{lem}

 We turn now  to the main result.
 Three main problems may arise.

\begin{itemize}
\item  There may be  roots on the boundary of the considered set.

\item The  roots  $X(t)=u$ could be associated  with a Jacobian
     $X'(t)$,  that is  almost singular  or even singular, i.e., $\Delta(t) =0$.
\item  The quantities under consideration  may be too large, leading to non-integrability. 
\end{itemize}

The first problem will be overcome  by  applying Lemma
\ref{l:bufield}, while the last two will be resolved using a  {\bf bounding and tapering
argument}, followed by a monotone convergence  argument. This (without
tapering) is, for example, a routine argument  in stochastic calculus, 
where one often  localizes, for instance, stopping times.  Tapering
is necessary to maintain continuity, which is a key argument.  We introduce our notation.
\medskip

Let $\mathcal{F }$ be the  continuous monotone, piecewise linear
function: $ \R^+ \to [0,1]$  such that:  (i)   it vanishes on $[0,
1/2]$; (ii) it takes the value 1 on $[1,+\infty)$.  
Consider  $F_n (x)
:= \mathcal{F} (nx)$, a tapered version of the indicator function
$\ind_{x>1/n}$, and $G_n(x) := 1-\mathcal{ F}(x/n)$, a tapered version
of the indicator function  $\ind_{x<n}$

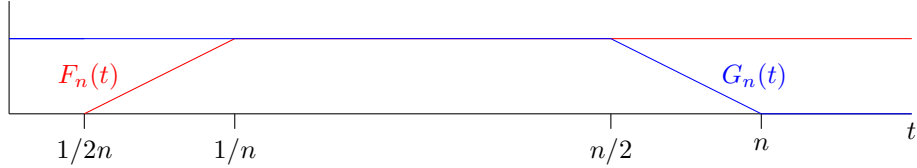
\begin{figure}[h]
\begin{center}
\begin{tikzpicture} 

\draw[-] (0,0) -- (12,0) node[below] {$t$};
\draw[-] (0,0) -- (0,1.5);
\draw[color=blue] (0,1) -- (1,1);
\draw[thin] (1,-.2) node[below] {$1/2n$} -- (1,0);
\draw[color=red] (1,0) -- (2,.5) node[left] {$F_n(t)\quad\,$} -- (3,1);
\draw[thin] (3,-.2) node[below] {$1/n$} -- (3,0);
\draw[color=red] (3,1) -- (12,1);

\draw[color=blue](0,1) -- (8,1);
\draw[thin] (8,-.2) node[below] {$n/2$} -- (8,0);
\draw[color=blue] (8,1) -- (9,.5) node[right] {$\quad G_n(t)$} -- (10,0);
\draw[thin] (10,-.2) node[below] {$n$} -- (10,0);
\draw[color=blue] (10,0) -- (12,0);
\end{tikzpicture}
\caption{Tapered functions}
\end{center}
\end{figure}

\begin{proof}[Proof of Theorem \ref{t:krf:dD} for $D=d$]


 For the moment,  consider $R$ to be a   compact  hyper-rectangle contained in $T$.  
  Let $v$ be a point in $\R^d$. 
 By applying of Lemma \ref{l:bufield} almost surely there is no root of $X(t) =v$  on the boundary of $R$ which has dimension $d-1$.
  
  The bounding  and tapering argument  is  introduced by considering 
\begin{align}
C^n_v(R) &:=  \sum _{ s \in R :  X(s) =v}
 F_n (\Delta(s)) G_n(\Delta(s))   ,\label{e:ccc}
   \\
Q_v^{n}(R) &:= C^n_v(R)  G_n (C^n_v(R)). \label{e:qqq}
\end{align}
In  (\ref{e:ccc}), when the summation index set is empty, we put  $
C_v^n(R)=0$. Let $g : \R^d \to \R ^+$ be continuous with compact
  support. We apply the area formula \cite[eq. (6.2), page 161]{AW-book} 
for the function
$$
h(t,v) = F_n( \Delta(t)) G_n ( \Delta(t))G_n(C^n_v(R)) g(v)
$$
to get that 
\begin{equation*}
\int_{\R^d}  g(v) Q_v^{n} (R) dv
 =  \int_R \Delta(t) F_n
(\Delta(t)) G_n ( \Delta(t))   G_n (C^n_{X(t) }(R))~g( X(t)) dt.
\end{equation*}

Taking expectations in both sides and applying Fubini's theorem gives
  $$
 \int_{\R^d}  g(v) ~\E( Q_v^{n} (R) )dv =
\int_{\R^d}  g(v)\, dv   \int_R  \E \big[ \Delta(t) F_n (\Delta(t))
   G_n (C^n_v(R)) \s X(t)=v \big] ~ p_{X(t)} (v)   dt.$$
   Note that the conditional expectation  is well-defined, since the variable is bounded. 
Since this equality holds for any  continuous function $g$ with bounded
support, it follows that:
\begin{equation}\label{e:ppu}
  \E (Q_v^{n} (R))  = \int_R \E \big[ \Delta(t) F_n (\Delta(t) ) G_n
  (\Delta(t) ) 
    G_n (C^n_v(R)) \s X(t) =v\big]  p_{X(t)} (v)   dt,
    \end{equation}
for almost every $v$.

\medskip

%
%
Let us prove that the
left-hand side of (\ref{e:ppu}) is a continuous function of $v$.
Outside the compact  set 
$$
 \{ t\in R :  \Delta(t)
  \geq 1/2n\},
$$
the contribution to the sum (\ref{e:ccc}) defining $C^n_v(R)$ is
zero for any $v \in \R^d$. 
Now, consider a particular level $u$. By the local inversion theorem, the
number of points $t \in R $ such that
 $X(t) = u$ and  $\Delta(t) \geq 1/2n $, is finite, and all these
 points are interior points. Let $k$ denote this number.

 If $k$ is non-zero, then $X(t)$  is locally invertible in the $k$ neighborhoods
 $V_1, \ldots ,V_k \subset R$ around these $k$ points.  For $v$ in some
 (random) neighborhood of $u$, there is exactly one root of $X(s) =v$ in
 each $V_1,...,V_k$, and the contribution to $C_v^{n}(R)$ from these
 points can be made arbitrarily close to the contribution corresponding to $v=u$.
 Outside the union of $V_1,...,V_k$, $X(t) - u$ is bounded away from
 zero in $R$, so that the contribution  to $C_v^{n}(R) $ vanishes  if $v$
 is sufficiently close to $u$.

Therefore,  the function $v \to Q_v^{n}(R)$ is a.s.
continuous at $v= u$. On the other hand, it is clear from its
definition that $ Q_v^{n}(R) \leq n $, and an application of the
Lebesgue Dominated Convergence Theorem implies the continuity  of
$ \E(Q_v^{n}(R))$ as a function of $v$.

Let us examine the right-hand side of (\ref{e:ppu}).  The
key point is that, due to the bounding and tapering argument,
irregular points  do not contribute  to $C_v^n$. 
 As a consequence,  we can apply Proposition  \ref{p:poly}  to $C_v^n$.
 Recall that (a)  is satisfied. As for  (b) it is unnecessary because
 roots with a singular  jacobian $X'(t)$ are excluded  by the term $F_n
 (\Delta(t))$.
Hypotheses (ii) and (iii), along with the
boundedness of   the quantities considered due to the presence of $G_n$
imply that $$
\E \big[ \Delta(t) F_n (\Delta(t) ) G_n (\Delta(t) ) ) G_n (C^n_v(R))
\big / X(t) =v\big]
$$ 
is a bounded, continuous function of $v$. 
Since $p_{X(t)}(v)$ is continuous  as a function of $v$ and uniformly bounded
in $t$,  the right-hand side of (\ref{e:ppu}) is continuous and equality
in  (\ref{e:ppu}) holds {\bf  for every} $v$.

We can now take the limit as $n \to +\infty$, applying  the monotone
convergence theorem. This leads to the equivalent  of  \eqref{eq:KRF}
for the number of regular roots  in $R$. By applying
 Proposition \ref{p:buli}, for $D=d$, we get \eqref{eq:KRF} for $B=R$. 
Since the class of open relatively compact rectangles generates  the
Borel $\sigma-$algebra, and both members of \eqref{eq:KRF} are measures
on Borel subsets $B$ of $T$, a classic extension argument allows us extend the result. Thus, we obtain \eqref{eq:KRF} for  every Borel set  $B$.
\end{proof}

%

\section{Proof of Theorem \ref{t:krf:dD} for $D>d$}\label{sec:KRFD}

Let us start with some preliminaries on integral geometry.


 Given $\ell<D$,   we
consider the Grassmannian  manifold $\grass D\ell$ of $\ell$-dimensional
linear subspaces of $\R^D.$

Let  $d\grass D\ell$   be
the Haar measure: the uniquely defined orthogonally invariant probability measure on this space. 
Let $B$ be a Borel set in $\R^D$.
We define the $m$-\emph{integral geometric measure} of $B$
 by
 \begin{equation}\label{eq:igm1}
   \mathcal{I}_{D, m}(B):=c_{D,m}  \int_{V\in  \grass{D}{D-m}} d\grass{D}{D-m}(V) 
 \int_{y\in V^\perp} d\lambda_{m}(y) \,
 \#   \left\{ B\cap \ell_{V,y} \right\} ,
  \end{equation}
where $\ell_{V,y}$ is the affine linear space $\{y+V \}$,  and where  we
naturally identify the element of the Grassmanian $V\in\grass{D}{D-m}$ with the associated
subspace of codimension $m$ on $\R^D$. 
 The constant 
 $
 c_{D,m}$ is defined by 
\begin{equation}\label{constante-Crofton}
  c_{D,m} := \pi^{1/2} \frac{\Gamma\left(\frac{D+1}{2}
  \right)}{\Gamma\left(\frac{m+1}{2} \right)\Gamma\left(\frac{D-m+1}{2}
  \right)}.
\end{equation}
 
 The integrand  in \eqref{eq:igm1} is measurable
 (see, for example, Morgan \cite[p. 13]{Morgan}), and since it is non-negative, the integral is always well defined, finite, or infinite.
 This measure is also called the \emph{Favard measure}.
 
 Next theorem is a fundamental  formula of integral geometry.
\begin{thm}  [Crofton's formula]\label{t:crofton2}
If $B\subset\R^D$ is $m$-rectifiable, then Crofton's formula \cite[p. 31]{Morgan} yields
\begin{equation}\label{Crofton} 
\sigma_{m}(B)= \mathcal{I}_{D,m}(B).
\end{equation}\qed
\end{thm}

A simple way to characterise  $m$-rectifiable sets  is  to describe them
as  the union  of  countably many  $C^1$ manifolds and a set  of
$\sigma_m$  measure zero. 

 A  Borel set can be decomposed    into  the disjoint union  of a
 rectifiable set  and a purely non-rectifiable set. This last class of
 sets has always a  null integral geometric measure, and this implies that the Hausdorff measure
 of a set  is always greater or equal than the integral geometric
 measure. (For more details, and the definition of purely non-rectifiable
 set see \cite[3.17 Structure Theorem]{Morgan}.)

\begin{proof}[Proof of Theorem \ref{t:krf:dD}]
 First, note that under our hypotheses and  because of Proposition \ref{p:buli}, a.s. the level set  $\lev_u$ is rectifiable. 
The proof is a direct application of Formula \eqref{eq:igm1} and
  Theorem  \ref{t:krf:dD} for $D=d$. 
  By Fubini's theorem
\begin{multline*}
  \E \left( \sigma_{D-d}(\lev_u(B))\right) 
  =c_{D,D-d}\int_{V\in \grass{D}{d}}  \grass{D}{d}d(V)    \int_{y\in V^\perp}  d\lambda_{D-d}(y)
\, \E\#   \left\{ \lev_u(B))\cap \ell_{V,y} \right\},
\end{multline*}
where 
$$
 \E\#   \left\{ \lev_u(B)\cap \ell_{V,y} \right\} =\E\Big( N_u\big( X, B\cap (V+y)\big) \Big).
$$

    Let $V$ be as above, and let $v_1,\ldots,v_d$ be an orthonormal basis
    of $V$.  Let $\pi_V^\top$ be the linear map given by the matrix
  $[v_1,\ldots, v_d]$.

 We can parametrise  the affine space $(V+y)$ as 
 $$
 (V+y) = \{ (\pi_V) ^\top  s +y,\, s\in \R^{d}\}.
 $$
Now it is clear that the restriction of our field $X$ to the above affine space
  satisfies the conditions of Theorem \ref{t:krf:dD}, for the case
  $D=d$. Then, applying  the \KRF~ for $D=d$, and  Fubini's theorem, we
  obtain
 \begin{multline*}
  \E \left( \sigma_{D-d}(\lev_u(B))\right) \\
   =c_{D,D-d} \int_{V\in \grass{D}{d}}d\grass{D}{d}(V) \int_B 
 \E\big(|\det( X'(t)) (\pi_V)^{\top} |\, \s\, X(t) =u\big) p_{X(t)}(u) dt,
\end{multline*}
 After a new inversion of integral  we get the resulting identity 
$$
c_{D,D-d}   \int_{V\in G(D,d)}|\det(X'(t) (\pi_V)^\top)| d\grass{D}{d}(V)
  =  \Delta(t),
$$
which follows directly  from  Proposition \ref{prop:intgeo} below.
\end{proof}

 We begin with a well-known lemma about the computation of Gaussian determinants. 
 Although it is already known (see for example Anderson \cite{And}), we include the proof here for ease of reading.

 \begin{lem} \label{l: det:mario}
 Let $M$ be a $D\times d$ standard Gaussian  matrix (its entries  are i.i.d. N(0,1) r.v.), $D\geq d$. Then,
 $$
  \E  (\Delta(M) )= \E\sqrt{\left|\det MM^\top \right|} =  
   2^{d/2}  \frac{\Gamma \left( \frac{D+1}{2}\right)}{\Gamma
   \left(\frac{D-d+1}{2}\right)} =:  L_{D,d}.
 $$
 \end{lem}

\begin{proof}
 Let us write $M^\top=[R_1, \ldots, R_d]$.
  We use  an interpretation of  $  \Delta(M)$ as  the $d$-volume of the parallelotope generated by $ R_1,\ldots,R_d$: 
  $$  |\Delta( M)|=   \|R_1\| \prod_{j=2}^d \rm{dist}(R_j,S_{j-1}),
$$
with $ S_{j-1} = \rm{Span}( R_1,\ldots, R_{j-1})$. 
 Because of the invariance properties $\rm{dist}(R_j,S_{j-1})$, is distributed as a $\chi(D-j+1) $  with expectation
   $\sqrt{2} \frac{ \Gamma\left(\frac{D-j+2}{2}\right)}{
     \Gamma\left(\frac{D-j+1}{2}\right)}$ and all these variables are independent.  This  yields 
 \begin{equation} \label{e:det:mat}
   \E  \big(\Delta (M)\big)  =  2^{d/2}  \frac{\Gamma
   \left(\frac{D+1}{2}\right)}{\Gamma\left(\frac{D-d+1}{2}\right)}=  L_{D,d}.
 \end{equation}
 \end{proof}


   \begin{prop}\label{prop:intgeo}
  Let $M\in\R^{d\times D}$. 
  Then,
  \begin{equation}\label{e:croft:M}
     \sqrt{\left|\det MM^\top \right|}= c_{D,D-d}\cdot
     \int_{V\in \grass{D}{d}}|\det(M (\pi_V)^\top)| d\grass{D}{d}(V).
     \end{equation}
   \end{prop}
   
   \begin{proof}
   Using singular value decomposition and   invariance by isometry of the Haar measure and the determinant,  we can assume, without loss of generality, that 
$$
   M =\begin{pmatrix}
     \lambda_1 &  & &\\
     & \ddots & &0_{d\times(D-d)}\\
      &  &\lambda_d &
   \end{pmatrix}   
   $$
   where  the  $\lambda_i$'s are the singular values of $M$. 
   Using  the  linearity  of the determinant, we can even  assume that 
     $M = \begin{pmatrix}
       I_d & 0_{d\times(D-d)}
     \end{pmatrix} $ so that \eqref{e:croft:M} is proved up to the multiplicative constant.  To determine it, we use  the particular case of a standard Gaussian matrix $M$. 
   
   Concerning the left-hand side, Lemma \ref{l: det:mario} implies directly that 
   $$ \E (\Delta(M)) = L_{D,d}.$$
   
    As for the second,   the rows of $  M (\pi_V)^\top $  are the columns of $   \pi_V M^\top$, and they consist of $d$  independent standard Gaussian  vectors in $V$ .  Applying  Lemma \ref{l: det:mario}  in $V$ again, we get 
    $$
    \int_{V\in \grass{D}{d}}|\det(M (\pi_V)^\top)| d\grass{D}{d}(V) = L_{d,d} = 
     \frac{\Gamma \left(\frac{d+1}{2}\right)}{\sqrt{\pi}},
   $$ 
    giving 
    $$
     c_{D,D-d} = \pi^{1/2}  \frac{\Gamma
     \left(\frac{D+1}{2}\right)}{\Gamma \left(\frac{D-d+1}{2}\right)
     \Gamma\left(\frac{d+1}{2}\right)},
$$
which agrees with (\ref{constante-Crofton}).
\end{proof}

\section{Some extensions and applications}

Theorem 1 has many extensions. We can consider  (a)  the ``non-flat
case'', where the  parameter set  $T$  is now a $D$-dimensional manifold,
such as a sphere; (b) the computation of higher moments; and
(c)  the addition of weights to the roots. 
Naturally, all combinations of (a), (b), and (c) can also be explored. 
We provide a non-exhaustive list of statements in the appendix.

Here, because  we need  it in the examples, we consider the  case of
weights (c).  Suppose we want to compute  an integral over  the
level set  with a  non-negative weight $g(t,Z(\cdot))$, which depends on the location $t$  and some extra random field $ Z(\cdot)$. 
We will consider the following:
 \begin{equation} \label{e:g} G_u(X,B)  =  \int _{t\in T: X(t)
=u} g(t,Z(\cdot)) d \sigma_{D-d} (t).   \end{equation} 
What types of weights are admissible? Let us provide two examples:
$\bullet$ Up-crossings. Let $ D=d=1$, $Z(t) = X'(t)$, $g(t,Z(\cdot)) =
\ind_{Z(t)>0}$. 

$\bullet$ Critical points  with a given index: Let  $Y(t)$, where $t\in T
\subset \R^d$, be a $C^2$ random field with real values. We  define
$X(\cdot)$ as its gradient. The critical points of $Y$ are the zeros of
$X(\cdot)$.  We define the index  $i(M)$ of a matrix as the number of its
negative eigenvalues. Among the critical values  mentioned, we may want to
count  only those  associated   with a Hessian of  index $k$. In this
case, we need to define: $Z(\cdot)=X'(\cdot)=Y''(\cdot)$,
$g(t,Z(\cdot))=\ind_{i(Z(t))=k}$. 
\bigskip
  
  We assume the following:  
  \begin{itemize} 
  \item[a)]  $g$ is  lower semi-continuous as a function of $t$;       
   \item[b)]  $g$ is   lower semi-continuous as a function of $Z(\cdot)$ for the  weak topology   on the space  $C^0(T)$  of continuous functions;
   \item[c)] for {\bf every} $v \in N$ and every $t\in T$, the distribution  of $\{ Z(\cdot)\}$ conditional to  $X(t) =v$  is well defined as a probability  and is continuous, as a function of $v$ with respect to 
   the weak topology on the space  $C^1(T)$.
 \end{itemize}
   
\begin{thm} [Expected integral  on the level set]\label{t:w:dD}
 Let $X : T  \to \R^d$ be a random field  satisfying the hypotheses of
      Theorem \ref{t:krf:dD}. To each point $t$  of the parametric
      space, we associate  a weight $ g( t,Z(\cdot)) $, where  the
      function $g$ satisfies $a),\ b),\ c)$.
  Then,   \begin{equation} \label{e:poids} 
    \E \big( G_u(X,B) \big)  =\int_{B}\E\left( \Delta(t) g(t,Z(\cdot))
    \s X(t)=u\right) ~p_{X(t)}(u)dt,
 \end{equation}
 where $G_u(X,B)$  is defined in \eqref{e:g}.

\end{thm} 
  
\begin{rem} In the case  $X(\cdot) , Z(\cdot)$ are jointly
non-degenerate Gaussian, condition $c)$ is always satisfied. Admissible functions $g(\cdot)$ are often
indicator functions of an open set.
\end{rem}
\begin{proof} Since the function $g(\cdot) $ is a monotone limit  of a
sequence of continuous function, we can use a monotone convergence argument to assume that $g(\cdot) $ is  continuous.  Thus, the
conditions above ensure the continuity of both sides of  the formula.
\eqref{e:ppu}.\end{proof}

\section{Examples}\label{sec:examples}

\subsection{Critical points of Gaussian random fields } \label{s:g:stat}


Let   $X(\cdot) $ be a Gaussian random field  with $C^2$ paths  from
$T\subset \R^d \to \R$,  such that  for every $t \in T$,  the
variance-covariance matrix  $\Var(X'(t))$ is non-degenerate. Then, by
Proposition \ref{p:buli}, the sample paths are almost surely Morse  in
the sense that  $$  \P\{ \exists t \in T : X'(t) =0,  \det (X''(t) ) =0
\} =0.$$ So, we can define the modified (not considering boundaries)
Euler characteristic $\Phi:= \sum _{k=0}^d  (-1)^{d-k}\crit_{k,u} $,
where $\crit_{k,u}$ represents the critical points greater than $u$ of
index $k$ (see  \cite[lemma 11.7.11]{AT} and \cite{Es:Le}). When the
random field is Morse,  this quantity is equal, up to  boundary
problems,  to the  Euler characteristic of the excursion set:
$$ E_u :=
\{t \in T :  X(t) >u\}.  $$

Then, an easy  application of Theorem  \ref{t:w:dD} gives 
   $$ \E(\Phi) = (-1)^d\int_u^{+\infty} dx \int_T dt \ \E(\det(X''(t))\s
   X'(t) =0,X(t)=x) p_{(X'(t),X(t))}(0,x).$$

\subsection{ Function of Gaussian processes}
Before stating  our result, we need to establish  two lemmas.

\medskip

Let $Y \in \R^D$ be a random variable that admits a density $p(y)$ w.r.t.
   the Lebesgue measure. Let $f:\R^D \to \R^d$, $d \leq D$, be a $C^1$
   function. Define  $U= f(Y)$   and let $R_U$  and $I_U$  denote the sets
   of regular and  irregular values of $f(\cdot)$, respectively:
\begin{align*}
 I_U &:=  \{u\in \R^d : \exists y \in \R^D; u=f(y);\; \Delta(y) =0\},\\
R_U & :=  \{u\in \R^d : \exists y \in \R^D; u=f(y);\; \Delta(y) \neq 0\},
\end{align*}
where $\Delta(y) := \det\Big( f'(y) \big( f'(y) \big)^T\Big)$.
Note that $ I_U $ and $ R_U $ may have some intersection.

 \begin{lem} [Push-forward  measure]  \label{p:push-forward:g}  
The push-forward distribution, representing the distribution of $U= f(Y)$, is
   the sum of a measure supported on the set $I_U$  and  an
   absolutely continuous distribution on the set $R_U$ of regular values
   of $f$ with density 
\[h(u) =  \int _{y\in f ^{-1}(u)} \frac{ p(y)}{ \Delta (y)} d\sigma_{D-d} (y).  
\]
In particular, if  $\P \{\Delta(Y)  =0\} =0 $, then the distribution  of $U$ is absolutely continuous, with density given by the formula above.
\end{lem}
 The proof of this lemma is provided in  the appendix.

Note that  if $ f$ is $ C^{D-d+1}$, the Morse-Sard theorem (see
Hirsch \cite{Hirsch},  for example) implies that  $I_U$  has zero Lebesgue
measure. 
 \begin{lem} [Continuity  of integrals on the level set] \label{l:contour} 
   Let $K$ be a compact set in $ \R^D $, and let $H : K \to \R^d$, $D\geq d$ be a
   $C^1$ function. Let $g : \R^D \to \R^d$ be a continuous function,   and let
   $W$ be an open set in $\R^d$. Suppose that  for every $y \in W$
\[ \big\{ x \in H^{-1}(y) : \det(H'(x) H'(x)^\top )  =0\big\} =\emptyset ,\]
where  $ H^{-1}(y) := \{x \in K : H(x) = y\} $.

Then,   

\[ y \to \int_{ H^{-1}(y)}   g(x) d\sigma_{D-d}(x) \mbox{ is continuous on } W.
\]
\end{lem}  
For the proof, see Caba\~{n}a  \cite{cabana-web}.

\begin{proof}[Proof of Theorem \ref{c:foncgauss}]

Let us verify  the conditions of Theorem \ref{t:krf:dD}. Condition $(i)$ is clear.

 To prove the validity  of the Kac-Rice formula  for
  every Borel set $B$ contained in $T$,  it suffices to establish
  \eqref {e:krf:dD} for  every  compact subset $K \subset T$.
  Lemma  \ref{p:push-forward:g} implies that, for all $t\in T$,
  the density of   $ X(t)$ at $u$  is given by
  $$ p_{X(t)}( u)=\int_{ y\in H^{-1}
  (u)}    \frac{  \phi_t (y)} { \sqrt{ \det(H'(y) H'(y)^\top )}} \
  d\sigma_{n-d}(y),$$ 
 where  the preimage  is restricted to $K$. Here, $\phi_t(y)$ is the
  Gaussian density of $Y(t)$.  Since $H^{-1}(u)$ is compact, the determinant
  $\det(H'(\cdot) H'(\cdot)^\top )$ is lower bounded on this set.  By
  applying Lemma \ref{l:contour},  the  density $ p_{X(t)}(v) $ is
  continuous and bounded, thus condition $(ii)$ of Theorem  \ref{t:krf:dD} holds.


%
%
%

We now consider hypothesis $(iii)$. Let
  $\mathcal F$  be a bounded continuous  functional  on $C^1(K)$. Since
  such a functional can be expressed as the difference of two non-negative
  functionals, we restrict our attention  to the case  of a non-negative
  functional.
  Our aim is to prove that 
$$\E ( \mathcal F(X(\cdot))\s X(t)=v) \mbox{ is continuous as a function of } v.$$


  We know that the conditional expectation $   \E(\mathcal F(X(\cdot)) \s Y(t)  =y) $ is well-defined,  and that 
$$\E(\mathcal F)  = \int_{\R^n}    \E(\mathcal F(X(\cdot)) \s Y(t)=y) p_{Y(t)} (y)  dy.
$$

    Using the generalized change of variable $H(y)=x$ (i.e., the co-area formula),
$$ \E ( \mathcal F)  = \int_{\R^d} dx \int_{ H^{-1}(x)}    \frac{  \E(\mathcal F(X(\cdot)) \s Y(t) = y)} { \Delta(y) }  p_{Y(t)} (y) \sigma_{n-d}(dy), $$
where $\Delta(y) := \sqrt{ \det(H'(y) H'(y)^\top )}$.

 Consequently,  we define a version of the conditional expectation as:
 $$ \E (\mathcal F(X(\cdot))\s  X(t) = x )  := \frac {1} {  p_{X(t)}(x)} \int_{ H^{-1} (x)}    \frac{  \E(\mathcal F(X(\cdot)) \s Y(t) = y)} { \Delta(y) }    p_{Y(t)} (y) \sigma_{n-d}(dy).
$$

 Note that this integral is always well-defined. By assumption,  $ \E(\mathcal F(X(\cdot)) \s Y(t) = y)$ is continuous.  By $(c)$,
  $\Delta(y)$ is continuous  and    non-zero almost everywhere. Applying  Lemma \ref{l:contour}, we conclude that the mapping $$x \to \int_{ H^{-1}(x)}
  \frac{  \E(\mathcal F(X(\cdot)) \s Y(t) = y)} { \Delta(y) }    p_{Y(t)} (y)
  \sigma_{n-d}(dy)$$ is continuous.   
\end{proof}

The  classical example is, of course, the case  of $\chi^2$  processes:
$$  X(t)  = \|Y(t)\|^2,$$ with $\Var(Y(t) )= I_n$ . Conditions a) and b) are straightforward; as for c) 
  $$
  H'(y) = 2(y_1,\ldots, y_n); \quad   \det(H'(y) H'(y)^\top ) = 4 \|y\|^2 = 4v.
  $$
  Therefore, excluding the trivial case $u=0$, the hypothesis is satisfied.

     In the proof of Corollary \ref{c:foncgauss}, the Gaussianity  of $Y(\cdot)$ is not strictly necessary; all that is required is the regularity of the density of $Y(t)$ and the regularity of the conditional expectation $ \E(\mathcal F \s Y(t) = y)$.  The result can be generalized, as in the next section. 

\subsection{Sum of a parametrized function}

In this section, we consider  the following general framework:
$$
X(t)  =  \sum_{i=1}^n  L(Y_i,t) := H_t(\mathbf Y),  \quad \mbox{ with } \mathbf Y := (Y_1,\ldots,Y_n ),
$$
with 
\begin{itemize}
\item $t \in T \subset  \R^D$;
\item  the $Y_i$'s are i.i.d. random variables in $\R^m$  with density $f(\cdot)$;
\item  $L(\cdot) $ is  a function taking values in $\R^d$, $d \leq D$, which is   $C^1$ as a function of $t$ and $Y$;
\item $nm\geq D$.
\end{itemize}

\begin{thm} \label{p:sum}
With  the notation and assumptions above, assume that for every $t\in T$,
  $v\in \R^d$, the irregular set 
 $$ \{ \bY \in   H_t^{-1} (v)  :  H_t'(\bY) \mbox{ is not of full rank }\} ,   $$
is almost surely empty.
 Then, the conditions of Theorem \ref{t:krf:dD} are fulfilled, and the
  \KRF~holds for $u\in \R^d$.
 
 \end{thm}


 \begin{proof}
  Let us verify  the hypotheses of  Theorem \ref{t:krf:dD}. 
   Condition $(i)$  is clear. 
   Condition $(ii)$ follows from the analysis of the density as shown in the proof of Theorem  \ref{c:foncgauss}. 
   It remains  to establish  condition $(iii)$. To achieve this, we will follow the lines of the proof of 
   Theorem \ref{c:foncgauss}. Fix a particular value of the parameter $t_0$, 
 and  consider a  bounded, continuous,  non-negative functional $\mathcal F$ on $C^1(T)$. We aim to define
   $\E (\mathcal F(X(\cdot)) \s  X(t_0) =v) $ as a continuous function of $v$.

   It is clear  that 
   $$
   \E(\mathcal F(X(\cdot))\s \bY) = \mathcal F( \mathcal Y),
   $$
   where 
    $ \mathcal Y$ is the function 
    $$   t \to H_t( \mathbf Y).
    $$
 This functional  is continuous  as a function of   $ \mathbf Y$ due to the differentiability  hypothesis. 
 We now  choose a version of the conditional expectation
     $$ \E ( \mathcal F(X(\cdot))\s  X(t_0)  = v)  := \frac {1} {  p_{X(t_0)}(v)}
     \int_{H^{-1} (v)}    \frac{  \E(\mathcal F(X(\cdot)) \s \mathbf Y = y)} {
       \Delta(y) }   p_{\mathbf Y} (y) \sigma_{n-d}(dy), $$
 where $H$ stands for $ H_{t_0}$ and $\Delta(y) $ represents its generalized
   Jacobian. 
      Since,  by hypothesis  $ \E(\mathcal F (X(\cdot))\s \mathbf Y = y)$ is continuous, applying  Lemma \ref{l:contour} gives us that,
   $$
   v \to \int_{H^{-1} (v)}    \frac{  \E(\mathcal F(X(\cdot)) \s  \mathbf Y = y)} { \Delta(y) }    p_{\mathbf Y} (y)\sigma_{n-d}(dy)
$$
   is continuous. Consequently, $\E ( \mathcal F(X(\cdot))\s  X(t_0)  = v) $ is a continuous function of $v$.   This proves $(iii) $ and achieves the proof:
 $$
   \E(\sigma_{D-d}( \lev_u ))=  \int_T \E(\Delta(t)\s  X(t) =u \big) 
  p_{ X(t)}(u) \ dt.
   $$
   \end{proof}

\subsection{Mean number  of  critical points of the likelihood}
 An example of a  situation corresponding to the framework above is the case of number of critical points of the likelihood function. 
 We consider  a parametric  statistical model. The observation   $\bY  := ( Y_1, \ldots, Y_n)$ consists of $n$ i.i.d.  observations $Y_i$ that  lie in $\R^m$.
 
 Let $L(Y,\theta) $ be the log-likelihood  associated  to the observation of  one $ Y \in \R^m$.
 We define
 $$
 L( \theta)  := \sum_{i=1}^n  L(Y_i,\theta).
 $$
 
 We assume that the model is parametric, meaning that $\theta$ varies in an open set  $\Theta$ of $\R^d$. Let us denote  the gradient of
 $L(\theta) \in \R^d$ by $L'(\theta)$:
 $$L'(\theta) = \frac{\partial}{\partial \theta }  L( \theta).
 $$
 
  We have 
  $$
   L'( \theta)  := \sum_{i=1}^n  L'(Y_i,\theta).
$$
Under the appropriated conditions, we can apply 
 Theorem \ref{p:sum} to  compute  the expectation  of the number $N$ of critical points of the likelihood. 
  $$
  N := \#\{ \theta \in \theta :  L'(\theta) =0 \}.
  $$

\subsection{Gravitational stochastic microlensing}
 
 This paragraph is mainly based on the paper of Petter et al.
 \cite{Pe2}. 
 Let us consider a gravitational potential $\psi:T\subset \mathbb R^2\to \mathbb R$,  and  define the transformation :
\begin{equation}
\eta(x)=x-\psi'(x).
\end{equation}
A lensed image is defined as  the point solution of the following equation
$\eta(x)=y.$ 

From now on, we will focus on  a type of lens system known as
stochastic microlensing. For the definition, let us consider a number $N$ of stars,
each located at a random position $\xi_j\in\R^2,\,j=1,\ldots,N.$ In this case, the
above quantities are  defined as follows: \begin{eqnarray*}
  \psi(x)&=&\kappa_c\frac{|x|^2}2-\frac{\gamma}2(x_1^2-x_2^2)+m\sum_{j=1}^N\log\|x-\xi_j\|
  \ ,\\ \eta(x)&=&(1-\kappa_c+\gamma)x-2m
  \sum_{i=1}^N\frac{x-\xi_i}{\|x-\xi_i\|^2}\ , \end{eqnarray*} where
  $x=(x_1,x_2)\in\R^2-\{\xi_1,\ldots,\xi_N\}$,  $\gamma$ is the external
  shear,  in continuous matter with constant density $\kappa_c\ge  0$,
  $m$ the mass of the stars and $N$ the number of stars, that remain
  fixed. We also use the condition  $(1-\kappa_c+\gamma)<0$, which is
  known as the supercritical case in the physical literature see
  \cite[pp. 434]{Pe1}. We assume that the positions of the stars are
  independent random vectors $\xi_j=(U_j,V_j)$   with  a density $f$ supported 
  on the ball $D=B_2(0,R)$.  The lensed image with $y$,  corresponding to the light
  source, is given by the equation $$\eta(x^*)=y\mbox{ which can be expressed as } 2m
  \sum_{i=1}^N\frac{x^*-\xi_i}{\|x^*-\xi_i\|^2}-(1-\kappa_c+\gamma)x^*=y.$$ 

We will consider the expectation of the number of images produced by  the source $y$ in a region $D\subset\R^2$,  which we will denote by $N(D,y)$.

To place this problem within the framework of the previous development,  we introduce the function $L(\mathbf\Xi,x): D\subset \mathbb R^2\to\mathbb R^2$ through the expression
$$
L(\mathbf\Xi,x)
=\sum_{i=1}^N\left[\frac{(1-\kappa_c+\gamma)x}N-2m\frac{x-\xi_i}{\|x-\xi_i\|^2}\right]\mbox{ where } \mathbf\Xi=(\xi_1,\ldots,\xi_m).
$$

The function $L(\mathbf\Xi,x)$ has $m$ singularities at the points $x=\xi_j$, but is a $C^{\infty}$ function
in $\R^2-\{\xi_1,\ldots,\xi_N\}$. 
Let $y\in\R^2$, and let $\mathbf V$ be a neighborhood of $y$.
Since we have $\lim_{x\to\xi_j}\|\eta(x)\|=+\infty$, all the roots
$\eta(x)=y'$ for $y'\in\mathbf V$ are located $\R^2-\{\xi_1,\ldots,\xi_N\}$.
Therefore, in the set $\eta^{-1}(\mathbf V)$, the function $ \eta$  is
$C^{\infty}$, which confirms condition $(i)$ in the localized version of Theorem
\ref{t:krf:dD}.

Let us consider the following push-forward formula:
$$p_{L(\mathbf\Xi,x)}(y)=\int_{L^{-1}(\mathbf\Xi,x)(y)}\frac1{\Delta(\mathbf\Xi,x)}\prod_{j=1}^Nf(\xi_j)d\xi_j \ ,
$$
where  
$$
\Delta(\mathbf\Xi,x)=|\det\frac{\partial L(\mathbf\Xi,x)}{\partial x}|.
$$ 

Hence, to verify conditions  $(ii)$ and  $(iii)$ of Theorem  \ref{t:krf:dD}, we
will utilize  Theorem \ref{p:sum}.  Let us examine the condition of this proposition. Specifically, we need to analyze
the behavior of $\Delta(\mathbf\Xi,x)$.
Hence,  we  first compute 
\begin{multline*}\frac{\partial L(\mathbf\Xi,x)}{\partial x}\\=(1-\kappa_c+\gamma)I+\begin{bmatrix}-2m\sum_{i=1}^N(\frac1{\|x-\xi_i\|^2}-2\frac{(x_1-\xi^1_i)^2}{\|x-\xi_i\|^4})&
2m\sum_{i=1}^N\frac{(x_1-\xi^1_i)(x_2-\xi^2_i)}{\|x-\xi_i\|^4}\\
2m\sum_{i=1}^N\frac{(x_1-\xi^1_i)(x_2-\xi^2_i)}{\|x-\xi_i\|^4}&-2m\sum_{i=1}^N(\frac1{\|x-\xi_i\|^2}-\frac{(x_2-\xi^2_i)^2}{\|x-\xi_i\|^4}) \end{bmatrix}.\end{multline*}
In this form, we obtain
$$\langle\frac{\partial L(\mathbf\Xi,x)}{\partial x}
v,v\rangle=\Big[(1-\kappa_c+\gamma)-2m\sum_{i=1}^N\frac1{\|x-\xi_i\|^2}]\|v\|^2+2m\sum_{i=1}^N\frac1{\|x-\xi_i\|^4}\langle
x-\xi_i,v\rangle^2$$
$$=[(1-\kappa_c+\gamma)-2m\sum_{i=1}^N\frac{1}{\|x-\xi_i\|^2}(1-\langle\frac{x-\xi_i}{\|x-\xi_i\|},\frac{v}{\|v\|}\rangle^2)\Big]\|v\|^2.$$
When $\|v\|=1$, we have $$\langle\frac{\partial
L(\mathbf\Xi,x)}{\partial x}
v,v,\rangle=[(1-\kappa_c+\gamma)-2m\sum_{i=1}^N\frac{1}{\|x-\xi_i\|^2}(1-\langle\frac{x-\xi_i}{\|x-\xi_i\|},v\rangle^2)].$$

Since  $(1-\kappa_c+\gamma)<0$, for any vector $v$, it holds that $\langle\frac{\partial L(\mathbf\Xi,x)}{\partial x} v,v\rangle<0$. Thus, in this case,
we get $\Delta(\mathbf\Xi,x)>0$. Therefore, conditions  $(ii)$  and $(iii)$ of Theorem  \ref{t:krf:dD}  hold true. 

The KRF writes in this case
\begin{equation}\label{ricelensing}\E(N(D,y))=\int_D\E[|\det\frac{\partial L(\mathbf\Xi,x)}{\partial
x}|\,|\eta(x)=y]f_{\eta(x)}(y)dx,\quad\mbox{for all } y.\end{equation}
\begin{rem} In Theorem 4 of \cite{Pe2}, a version of a similar formula for almost every $y $ has been proved. Our method allows us
to obtain, in the supercritical case, the formula \textbf{for all} $y$. We believe that the formula is true for all $y$ in the subcritical case; however, the preceding procedure cannot be applied.\end{rem}
\begin{rem}Although finding an exact expression for the formula \eqref{ricelensing} is computationally complex, the cited reference demonstrates its usefulness in evaluating the asymptotic value as $N\to\infty$.
\end{rem}

 \subsection{Shot-noise }
Following \cite{Bier:De}, the shot-noise process $X=\{X(t)\colon t\in\R^d\}$ is defined by the equation
\begin{equation}\label{eq:sn}
X(t)=\sum_i\beta_i g((t-\tau_i)),
\end{equation}
where the  $(\beta_i)$'s are  a sequence of i.i.d. random
real variables (the ``impulse") and $(\tau_i)_i$ is a Poisson field with  intensity (say)   1.

We assume  the following conditions:
\begin{itemize}
\item [(H1)]The kernel $g(\cdot)$  is  of class $C^1$   with  compact support included in $(-\eta,\eta)^d$ .
 \item[(H2)]  The impulse $\beta$  has a finite  first moment.
 \item[(H3)]  The product $\beta_1 g(\mathscr T)$  has a bounded density, where
   $\mathscr T$ is a random variable with uniform distribution over  $(-\eta,\eta)^d$.
 \end{itemize}

Let us consider a  fixed instant $t$.  Without loss of generality, we can assume that $t=0$.
  The number $\mathcal P$ of   indices $(i)$  in $ \eqref{eq:sn}$  follows a  Poisson  distribution with parameter $(2\eta)^d$   and is  thus almost surely  finite.
 Let $u \neq 0$,  and denote by $\mathcal M_u$  the measure of the level set defined as
 $$
 \mathcal M_u := \sigma _{d-1} \{ t\in T: X(t) =u\},
 $$
for some $T\subset \R^d$.
 
 We can use the law of total probability  to express the expectation of the measure as follows:
\begin{equation}\label{e:tp}
\E(\mathcal M_u) = \E(\mathcal M_u\s \mathcal P =1 ) \P\{ \mathcal P =1\}+\cdots+    \E(\mathcal M_u\s \mathcal P =p )   \P\{ \mathcal P =p\}+\cdots
\end{equation}
 By definition  of the Poisson process, given that $\mathcal P =p $,
 the $ \tau_i$'s are i.i.d.  with  a uniformly distributed over
 $(-\eta,\eta)^d$. Each term of the sum can be computed by means of
 Theorem \ref{p:sum}.

  Let us check the  condition  of this proposition.  Under the condition $\{ \mathcal P =p\}$, the matrix
 $  H'_t$   is a  $(2p,1)$ matrix,   which is of full rank unless all its elements are zero.  A sufficient condition to avoid this scenario is,  for all $t$:
 $$
 \P\{ g(t-\tau) \neq 0\} =1,
 $$
 where $\tau$ is uniformly distributed on $(-\eta,\eta)^d$.
 As a conclusion, Theorem \ref{p:sum}. can be applied conditionally
 to $ \mathcal P$.

 Finally   by de-conditioning,  we obtain that the  usual
 expression:  \eqref{e:krf:dD}  holds true
 for  every  $u \neq 0$  as long as  the  conditional expectation has been defined by all the process above.

\begin{appendix}
\section*{}
\subsection {Extensions of Theorem \ref{t:krf:dD}}
\subsubsection{Extension to manifolds} 
 In the present form, the  following  theorem has no equivalent in the literature.

   Let $M$  be a $C^1$  manifold  of dimension $D$ embedded in
   $\R^{D'}$, $D'>D$ with   the induced Riemannian structure. Let
   $X(\cdot)$  be a random field  defined on $M$ that takes values in
   $\R^d$ (or in another $d$ dimensional Riemannian manifold).   
   For a Borel set $B$  in $M$ and for $u\in \R^d$, we define the level set  as 
   $$
  \lev_u := \{t\in B :  X(t) =u\}.
   $$
   Roughly speaking, this set  is of dimension $D-d$,  making it natural to consider  the  Hausdorff measure: 
   $$
  \sigma_{D-d}( \lev_u).
  $$
   By employing a partition of the unity argument, 
   the ``weak'' Bulinskaya  property   can be localized  within the
   domain of a local  chart,  allowing us to extend Proposition \ref{p:buli} to our case. 
    Consequently, the level set remains almost surely rectifiable,  and we can apply the framework of the proof of  Theorem
     \ref{t:krf:dD}. The main modifications are as follows:
     \begin{itemize} 
     \item Using  the area formula for manifolds \cite{Morgan}. 
    \item  Replacing  the compact hyper-rectangle $R$  with  a compact set, that is the closure of its interior. 
    \end{itemize}
     Thus, we obtain the following result:
          \begin{thm} \label{t:krf:dm} The result (regarding the validity of the \KRF), as stated in
            Theorem \ref{t:krf:dD}, can be extended to   a random
            field defined on  a 
            $C^1$ manifold $U$ of dimension $D$  and taking
            values in a $C^1$ manifold of  dimension  $d$.
            Define the
            generalized Jacobian  by $$ \Delta(t) := \sqrt{\det(\nabla
            X(t) \nabla X(t) ^\top)}, $$ where $ \nabla X(t)$   is the
            Riemannian gradient, defined by  differentiation  along
            curves.          Then, for every Borel set $B$ included in
            $U$, one has \begin{equation} \label{e:krf:man}  \E
            \left(\sigma_{D-d}\left(\lev_u(B) \right) \right) =\int_{B}
            \E\left( \Delta(t) \s X(t)=u\right)\,p_{X(t)}(u)d\sigma(t),
            \end{equation} where $\sigma$ is the surface measure  on
            $U$.  \end{thm}
  Note that, in practice,  the computation of $\det( \nabla X(t))$ should be
  performed using an orthogonal basis of the tangent space at the point $t$.

\subsubsection{Higher moments} 
The formula for  higher moments has a specific form  when $D=d$ and the considered levels are all equals. Using the notation of  Theorem \ref{t:krf:dD},
let $D=d$ and  let $k$ be an integer $k\geq 2$. The \KRF~  of order $k$ provides
the $k$th {\bf factorial moment} of order $k$:
\begin{multline}\label{e:krf:d:k}
\E \left[   \big( N_u(X,B) \big) \big(N_u(X,B)-1 \big) \ldots \big(N_u(X,B)-k+1 \big)
\right] \\
 =
\int_{B^k}
\E_C\left( \Delta(t_1) \ldots \Delta(t_k) )\right) ~p_{(X(t_1),\ldots,X(t_k))}(u,\ldots,u)dt_{1}\ldots dt_{k},
\end{multline}
where $\E_C$ is the expectation with respect to the condition:
 $$C=\{X(t_1)=u,\ldots,X(t_k)=u\}.$$

In the  cases  $D>d$  or when  all the considered sets $B_1,\ldots,B_k$  are pairwise different, or when the considered level $u_1,\ldots,u_k$ are distinct (we omit more complicated alternatives)  the \KRF~of order $k$ can be expressed as follows: 
 \begin{multline} \label{e:krf:d:k2}
   \E \left(\sigma_{D-d}\left(\lev_{u_1}(B_1) \right) ,\ldots,
   \sigma_{D-d}\left(\lev_{u_k}(B_k)  \right) \right)
 \\=
 \int_{B_1\times \cdots \times B_k} \E_C\left( \Delta(t_1),\ldots,
 \Delta(t_k ) \right)\,p_{(X(t_1),\ldots, X(t_k))}(u_1\ldots,u_k )dt_1\ldots dt_k. \end{multline}

  To understand this change of  form, we
  define the measure $\mathcal  M$ on $T^k$  as follows: 
   $$
   \mathcal  M(B_1\times \cdots \times B_k)
   :=\E(\sigma_{D-d}\left(\lev_{u_1}(B_1) \right),\ldots,
   \sigma_{D-d}\left(\lev_{u_k}(B_k) \right).$$
   
   Then,  the \KRF~of order $k$ provides the density of the   absolutely
   continuous part  of the measure.   The measure contains a
   singular part in the case $D=d$, but not   in the other case.

 \begin{thm}[Kac-Rice formula for higher moments] \label{t:krf:d:k}
Let $X : T  \to \R^d$ be a random field, where $T$ is an open subset of
$\R^d$,  and let $u \in \R^d$ be a fixed point. Assume the following: \\
\noindent(i)  The sample paths of $X(\cdot)$  are almost surely  $C^{1}$.\\
\noindent (ii')  For every pairwise distinct  $t_1,\ldots,t_k\in T$, $(X(t_1),
   \ldots, X(t_k))$ admits a continuous density $ p_{X(t_1), \ldots,
   X(t_k)}(v_1,\ldots , v_k)$, which is uniformly  bounded for
   $(v_1,\ldots , v_k)\in (\R^d)^k$ and $t_1,\ldots,t_k$ in any compact
   subset of $T^k\setminus \mathscr{D} $, where $\mathscr D$ is the diagonal in $T^k$: $$
   \mathscr D= \{(t_1,\ldots,t_k)\ \in T^k: \exists i\neq j,t_i = t_j\}.  $$
\noindent   (iii')  For {\bf every} $u_1,\ldots , u_k $,    for every pairwise distinct
    $t_1,\ldots,t_k\in T$, the    distribution  of $\{ X(s), s
   \in T\}$ conditional to  $ X(t_1)=u_1,\ldots,X(t_k) =u_k$  is well-defined as a probability  and is continuous as a function of
   $u_1,\ldots , u_k $ with respect to the weak topology  on the space
   $C^1(T)$. See  Proposition \ref{p:poly}.

Then, the equation   \eqref{e:krf:d:k} or \eqref{e:krf:d:k2}
   holds true.

\end{thm}

\subsection{Proof of Lemma \ref{p:push-forward:g} }

\begin{proof} Define 
\begin{align*}
 I_f &:=  \{y\in \R^D :  \Delta(y) =0\}\\
  R_f & :=  {I_f}^c.
 \end{align*}
The distribution of $Y$ (and, consequently, that of $U$) can be divided  into  two components: 
\begin{itemize} 
 \item One component is supported on $I_f$ and induces for $U$ a distribution on the set $I_U$ of singular values of $f(.)$; 
 \item The other component is supported on $R_f$ and induces a distributionfor $U$ supported on $R_U$, the set of regular values of $f(.)$. 
\end {itemize}

We study the distribution of $U$ restricted to the set $R_U$ of regular values.
Let $g(\cdot)$ be a bounded Borel function defined on $\R^d$ and with support on $R_U$. Then, we have
\[ 
    \E (g(U))  = \E \big(g(f(Y) )  \big)  =  \int_{R_Y} g(f(y)) \,  p(y)\,  dy.
\]
    Applying  the co-area formula  (see for example \cite{azais2005})
     yields 
\begin{align*}
  \E (g(U)) &= \int_{u\in \R^d}  \int_{y\in f^{-1}(u)} \frac{g(f(y))
 \, p(y)}{\Delta(y)}\,  d\sigma_{D-d}(y)
\\
  &=   \int_{\R^d}   g(u)du \int _{y \in f^{-1}(u) }  \frac{
    p(y) } { \Delta(y)}  d\sigma_{D-d}(y),
\end{align*}
which gives the desired result.
\end {proof}

\end{appendix}

\begin{acks}[Acknowledgments]
The authors would like to thank Rafael Potrie for his helpful comments on the 
Bulinskaya-type lemma.  
They would also like to express their sincere
gratitude to the two anonymous Referees 
 for their valuable feedback, which has significantly improved the final version of the paper.

\end{acks}


\begin{thebibliography}{}
 
 
 \bibitem{adler}
 Robert J. Adler.  The geometry of random fields. Wiley Series in Probability and Mathematical Statistics. John Wiley \& Sons, Ltd., Chichester, 1981.
 

\bibitem{AT}
Robert J. Adler and Jonathan E. Taylor. Random fields and geometry,
Springer Monogr. Math. Springer, New York, 2007. xviii+448 pp.

\bibitem{And} Theodore W. Anderson.
An introduction to multivariate statistical analysis.
Wiley Publications in Statistics
John Wiley \& Sons, Inc., New York; Chapman \& Hall, Ltd., London, 1958. xii+374 pp.


 \bibitem{angst2019zeros} Jürgen Angst and Guillaume Poly. On the zeros of non-analytic random
periodic signals. Int. Math. Res. Not. IMRN, no. 7, 4931–4968, 2022.

\bibitem{auffinger2013}
Antonio Auffinger, Gérard Ben Arous, and Jiri Cerny. Random matrices
and complexity of spin glasses. 
Comm. Pure Appl. Math. 66 (2013), no. 2, 165–201.

 \bibitem{Az} Jean-Marc Aza\"is. Approximation des trajectoires et temps local des diffusions. Ann. Inst. H. Poincaré Probab. Statist. 25, no. 2, 175–194, 1989.



\bibitem{azais2005}
Jean-Marc Aza\"is and Mario Wschebor. On the distribution of the maximum
of a Gaussian field with d parameters.  Ann. App. Probab.
15 (1A):254–278, 2005.

\bibitem{AW-book}
Jean-Marc Aza\"is and Mario Wschebor. Level sets and extrema of random processes and fields. John Wiley $\&$ Sons,  2009.

\bibitem{Be}
 Michael Berry. Statistics of nodal lines and points in chaotic quantum billiards: perimeter corrections, fluctuations, curvature. J. Phys. A 35(13), 3025–3038. 2002.
 
 \bibitem{berzin}Corinne Berzin, Alain Latour and José R. León. Kac-Rice formula: A contemporary overview of the main results and applications.
 arXiv preprint  arXiv:2205.08742, 2022.
 
 \bibitem{Bier:De}
 Hermine Bierm\' e and Agn\`es Desolneux.  On the perimeter of excursion sets of shot noise random fields. Ann. Probab. 44 , no. 1, 521–543, 2016.

\bibitem{bulinskaya}
Ekaterina V. Bulinskaya. On the mean number of crossings of a level
by a stationary Gaussian process. Theory Probability App.
6(4), 435–438, 1961.

\bibitem{Brillinger}
David R. Brillinger. On the number of solutions of systems of random equations. Ann. Math. Statist. 43, 534–540, 1972.

\bibitem{cabana}
Enrique M. Caba\~na. Esperanzas de integrales sobre conjuntos de nivel aleatorios, in Actas del segundo Congreso Latinoamericano de Probabilidad y Estadística Matemática (Caracas), 65-81, 1985.

\bibitem{cabana-web} Enrique M. Caba\~na. Integrals on random sets.\\
https://drive.google.com/file/d/1DyUHlV2FrUanUua6wBTUmIrJOhmvfT8x/view


\bibitem{CL}Harald Cramér and M. R.  Leadbetter. Stationary and related stochastic processes: Sample function properties and their applications. Reprint of the 1967 original.
Dover Publications, 2004.
\bibitem{Es:Le}Anne Estrade and José R. Le\'on. A central limit theorem for the Euler characteristic of a Gaussian excursion set. Ann. Probab. 44, no. 6, 3849–3878,
2016.
\bibitem{federer} Herbert Federer. Geometric measure theory.
Die Grundlehren der mathematischen Wissenschaften, Band 153
Springer-Verlag New York, Inc., New York, 1969. xiv+676 pp.

\bibitem{Grinberg} Eric Grinberg. On the smoothness hypothesis in Sard's theorem.
Amer. Math. Monthly 92 , no. 10, 733–734, 1985. 

\bibitem{Hirsch} Morris W. Hirsch. Differential Topology
Graduate Texts in Mathematics, no. 33. Springer-Verlag, 1976.

\bibitem{IZ} Ildar A. Ibragimov and Dmitry N. Zaporozhets. On the area of a
  random surface. Zap. Nauchn. Sem. S.-Peterburg. Otdel. Mat. Inst.  Steklov. (POMI), 384:154–175, 312, 2010.

\bibitem{Ito}
Kiyosi Itô. The expected number of zeros of continuous stationary Gaussian
processes. J. Math. Kyoto Univ. 3, 207–216, 1963/64.
\bibitem{Kac} 
Mark Kac. On the average number of real roots of a random algebraic
equation. Bull. Amer. Math. Soc. 49, 314–320, 1943.

\bibitem{KL}Marie Kratz and Jos\'e R. Le\'on. Central limit theorems for level functionals of stationary Gaussian processes and fields. J. Theoret. Probab. 14, no. 3, 639–672, 2001. 

\bibitem{longuet0}
David E. Cartwright and Michael S. Longuet-Higgins. 
The statistical distribution of the maxima of a random function
Proc. Roy. Soc. London Ser. A 237 (1956), 212–232.

\bibitem{longuet}
Michael S. Longuet-Higgins. The statistical analysis of a random, moving surface. Philos. Trans. Roy. Soc. London Ser. A 249, 321–387, 1957.

\bibitem{Morgan}
Frank Morgan. Geometric measure theory. Elsevier/Academic Press, Amsterdam, 2016.

\bibitem{NaSo}
Fedor Nazarov and Mikhail L. Sodin. Asymptotic laws for the spatial distribution and the number of connected components of zero sets of Gaussian random functions. Zh. Mat. Fiz. Anal. Geom. 12, no. 3, 205-278, 2016.

\bibitem{nualart-peccati}
David Nualart and Giovanni Peccati. Central limit theorems for sequences
of multiple stochastic integrals. Ann. Probab. 33(1):177–193, 2005.

\bibitem{Pe1}
Arlie O. Petters, Harold Levine and Joachim Wambsganss. Singularity theory and gravitational lensing.
Prog. Math. Phys., 21
Birkhäuser Boston, Inc., Boston, MA, 2001. xxvi+603 pp.

\bibitem{Pe2}
Arlie O. Petters, Brian C. Rider and Alberto Teguia.
A mathematical theory of stochastic microlensing. II: Random images, shear, and the Kac-Rice formula.
J. Math. Phys. 50, No. 12, 122501, 17 p. 2009.

\bibitem{Rice}
Stephen O. Rice. Mathematical analysis of random noise. Bell System Tech. J. 24:
46–156, 1945.

\bibitem{slud}
Eric Slud. Multiple Wiener-Itô integral expansions for level-crossing-count
functionals. Probab. Theory Related Fields, 87(3):349–364, 1991.

\bibitem{Whitney} Hassler Whitney. 
 A function not constant on a connected set of critical points. Duke Math. J. 1, no. 4, 514–517, 1935.



\bibitem{Wigman}
Igor Wigman. On the nodal structures of random fields: a decade of results.
J. Appl. Comput. Topol. 8 (2024), no. 6, 1917–1959.
 
 \bibitem{wschebor1982}
Mario Wschebor. Formule de Rice en dimension d.  Z. Wahrsch. Verw. Gebiete 60, no. 3, 393–401, 1982.

\bibitem{wschebor1985}
Mario Wschebor. Surfaces aléatoires: mesure géométrique des ensembles de
niveau. Lecture Notes in Math., 1147 Springer-Verlag, Berlin, 1985. xiii+111 pp.

\bibitem{Zahle}  Ulrich Zähle.
A general Rice formula, Palm measures, and horizontal-window conditioning for random fields.
Stochastic Process. Appl. 17, no. 2, 265–283, 1984.







 \end{thebibliography}
\end{document}